\documentclass[12pt]{article}
\usepackage{amsmath}
\usepackage{amssymb}
\usepackage{amscd}
\usepackage{amsthm}

\theoremstyle{plain}
\newtheorem{Thm}{Theorem}

\newtheorem{Cor}[Thm]{Corollary}
\newtheorem{Lem}[Thm]{Lemma}

\theoremstyle{definition}

\newtheorem{Expl}[Thm]{Example}

\theoremstyle{Remark}

\numberwithin{equation}{section}

\title{Remarks on the cone of divisors}
\author{Yujiro Kawamata}

\begin{document}

\maketitle

\begin{abstract}
We prove two theorems on the locally finite decompositions 
of the cones of divisors by the cones which correspond 
to canonical and minimal models.  
We introduce the concept of the numerical linear systems 
in order to simplify the argument on the Zariski decompositions.
\end{abstract}

\section{Introduction}

In the Minimal Model Program (MMP), a contraction morphism 
arises from an extremal ray on a {\em cone of curves}, a subset of a vector
space $N_1(X/S)$ consisting of numerical classes of effective $1$-cycles.
This is like a homology theory.
According to the later development of the MMP, 
it turned out that the subsets of the dual vector space $N^1(X/S)$, the 
vector space of divisors, are more useful.
This is like a cohomology theory.
In this paper we shall make some remarks on the finiteness properties
on subsets of $N^1(X/S)$.

Let $f: X \to S$ be a projective morphism between algebraic schemes.
Two $\mathbf{R}$-Cartier divisors $D_1$ and $D_2$, linear combinations
of Cartier divisors with real coefficients, on $X$ are said to be 
{\em numerically equivalent} over $S$, and denoted by $D_1 \equiv_S D_2$, 
or simply $D_1 \equiv D_2$,
if equalities $(D_1 \cdot C) = (D_2 \cdot C)$ hold
for all curves $C$ on $X$ which are mapped to points on $S$.
The set of all the numerical classes of $\mathbf{R}$-Cartier divisors
form a finite dimensional real vector space $N^1(X/S)$.
We consider the following inclusions of convex cones inside $N^1(X/S)$:
\[
\begin{CD}
\text{Amp}(X/S) @>>> \text{Nef}(X/S) \\
@VVV @VVV \\
\text{Big}(X/S) @>>> \text{Psef}(X/S)
\end{CD}
\]
where the {\em ample cone} $\text{Amp}(X/S)$ is the open convex cone
generated by the classes of Cartier divisors which are ample over $S$, 
the {\em nef cone} $\text{Nef}(X/S)$ is its closure,
the {\em pseudo-effective cone} $\text{Psef}(X/S)$ is the 
closed convex cone generated by the classes of effective Cartier divisors, 
and the {\em big cone} $\text{Big}(X/S)$ is its interior (cf. \cite{crep}).

We shall prove two theorems on the locally finite decompositions of 
these cones in \S 4 which are determined respectively by the canonical and 
minimal models (Theorems \ref{cano} and \ref{mini}).
The assertions of the theorems are basically contained 
in a paper by Shokurov \cite{Shokurov}.
But we make the statements more precise, 
especially in the second theorem.
We considered a partial decomposition of the vector space 
$N^1(X/S)$ by nef cones of various birationally equivalent 
minimal models in \cite{crep}.
That was a locally finite decomposition.
Shokurov's idea is to consider the space of boundary divisors themselves
instead of their numerical classes, so that
the assertions of theorems become not only local finiteness but also
global finiteness.
We treat log pairs whose boundaries are not necessarily big, but we assume that
there exist minimal and canonical models of the pairs (cf. \cite{BCHM}).
We include examples at the end of the paper 
on the whole space $N^1(X/S)$ where 
the finiteness holds only locally.

We introduce the concept of the numerical linear systems in \S 3 
which replaces that of the $\mathbf{R}$-linear systems in \cite{BCHM}
in order to simplify the argument on the Zariski decompositions.
We identify the exceptional divisors of the birational map to a 
minimal model as the numerical fixed part of the numerical canonical system 
(Lemma~\ref{exceptional}).
This lemma is used to prove the finite dcomposition theorem 
according to the choices of minimal models (Theorem~\ref{mini}).

%%%%%%%%%%%%%%%%%%%%%%%%%%%%%%%%%%%%%%%%%%%%%%%%%%%%%%%%%%%%%%

\section{Minimal and canonical models}

The purpose of this is section is to fix the notation and the terminology.
The MMP deals with pairs $(X,B)$ consisting of normal 
varieties and $\mathbf{R}$-divisors.
Let $\mu: Y \to X$ be a log resolution of the pair $(X,B)$,
i.e., $\mu$ is a birational morphism from a smooth variety $Y$
such that the inverse image of $B$ 
and the exceptional locus of $\mu$ 
as well as their union are normal crossing divisors.
We can write $\mu^*(K_X+B) = K_Y+C$ with irreducible decomposition 
$C = \sum_j c_jC_j$.
Then the pair is said to be {\em LC} (resp. {\em KLT}) if 
$c_j \le 1$ (resp. $c_j < 1$) for all $j$.
This definition does not depend on the choice of $\mu$.
The pair is called {\em DLT} if there exists a log resolution $\mu$ such that 
$c_j \le 1$ for all $j$ with strict inequalities when 
$\text{codim }\mu(C_j) \ge 2$.  

The cone theorem (\cite{KMM}) can be stated in the following way:

\begin{Thm}
Let $f: X \to S$ be a projective morphism from a DLT pair $(X,B)$.
Then the cone of nef divisors $\text{Nef}(X/S)$ in $N^1(X/S)$ 
looks locally rational
polyhedral when observed from the numerical class of $K_X+B$ in the 
following sense: let $V$ be the part of the boundary 
$\partial \text{Nef}(X/S)$ which is visible from the point
$v_0 = [K_X+B]$
\[
V = \{v \in \partial \text{Nef}(X/S) \,\vert\, 
[v,v_0] \cap \text{Nef}(X/S) = \{v\}\}
\]
then any compact subset in the relative interior of $V$ 
consists of finitely
many faces which are defined by linear equations with rational coefficients. 
\end{Thm}

Let $(X,B)$ be an LC pair consisting of a normal $\mathbf{Q}$-factorial 
variety and an $\mathbf{R}$-divisor
with a projective morphism $f: X \to S$ to a 
base space.
We assume an additional condition that there exists another boundary 
$\bar B$ such that $(X,\bar B)$ is a KLT pair.
Then the MMP or the MMP with scaling for the morphism 
$f: (X,B) \to S$ works and preserves 
the situation in the sense that the cone and contraction theorems hold and 
the resulting morphism after a divisorial contraction or a flip 
satisfies the same condition.
The existence of flips is proved in \cite{HM}, but the termination of flips 
in general is not yet.

A {\em minimal model} for a projective morphism $f: (X,B) \to S$ 
is defined to be a projective morphism $g: (Y,C) \to S$ 
with a birational map $\alpha: X \dashrightarrow Y$ over $S$
which satisfies the following conditions:

\begin{enumerate}
\item $\alpha$ is surjective in codimension one, $C = \alpha_*B$ is the
strict transform, and $(Y,C)$ is a $\mathbf{Q}$-factorial LC pair.

\item $K_Y+C$ is nef over $S$.

\item Any exceptional prime divisor for $\alpha$ has reason to be contracted in
the sense that it has positive coefficient
in the difference of $K_X+B$ and $K_Y+C$.
More precisely, if $p: V \to X$ and $q: V \to Y$ are common resolutions 
such that $q=\alpha \circ p$, then 
$p^*(K_X+B)-q^*(K_Y+C)$ has positive coefficient at the strict transform 
of an arbitrary prime divisor on $X$ whose image on $Y$ has higher codimension.
\end{enumerate}

The condition 3 is referred to as the {\em negativity} of $K_X+B$ for 
$\alpha$.

If a minimal model exists, then $K_X+B$ is pseudo-effective.
The converse, the existence of a minimal model, is still a conjecture.
The conjecture is proved to be true if the boundary $B$ is 
big in \cite{BCHM}. 

A {\em canonical model} is defined to be a projective morphism $h: Z \to S$ 
with a surjective morphism $g: Y \to Z$ with connected geometric fibers 
from a minimal model such that 
$K_Y+C = g^*H$ for an $\mathbf{R}$-Cartier divisor $H$ on $Z$ which is 
ample over $S$.

Minimal models are not necessarily unique 
for a given morphism $f: (X,B) \to S$.
But any minimal models are equivalent
in the following sense: if $\alpha_i: (X,B) \dashrightarrow (Y_i,C_i)$ 
for $i=1,2$ are minimal models and $p_i: V \to Y_i$ are common resolutions 
such that $p_1^{-1} \circ \alpha_1 = p_2^{-1} \circ \alpha_2$, then 
we can prove that 
$p_1^*(K_{Y_1}+C_1)=p_2^*(K_{Y_2}+C_2)$ as a consequence of 
the Hodge index theorem.
The negativity condition implies that 
the $Y_i$ are isomorphic in codimension one.
On the other hand, the canonical model is unique in the following sense:
the composite rational map to the canonical model 
$\beta = g \circ \alpha$ is uniquely determined by the given morphism 
$f: (X,B) \to S$. 

%%%%%%%%%%%%%%%%%%%%%%%%%%%%%%%%%%%%%%%%%%%%%%%%%%%%%

\section{Numerical linear system}

We propose definitions of numerical linear systems and 
numerical Zariski decompositions.
They are easier to deal with 
than $\mathbf{R}$-linear systems as in \cite{BCHM} and 
sectional decompositions as in \cite{crep} and \cite{Nakayama}. 

Let $f: X \to S$ be a projective morphism from a $\mathbf{Q}$-factorial
normal variety, and $D$ an $\mathbb{R}$-Cartier divisor 
which is pseudo-effective over $S$, i.e., $[D] \in \text{Psef}(X/S)$.
We define the {\em numerically fixed part} of $D$ by
\[
N(D) = \lim_{\epsilon \downarrow 0} 
(\inf \{D' \,\vert\, D' \equiv D +\epsilon A , D' \ge 0\})
\]
where $A$ is an arbitrarily fixed ample divisor, $D' \ge 0$ means that
the $\mathbf{R}$-divisor $D'$ is effective, and the infimum of 
$\mathbf{R}$-divisors is defined by the infimums of coefficients.
We note that $N(D)$ does not depend on the choice of $A$.
We also note that there are only finitely many irreducible components of 
$N(D)$, since they are linearly independent in $N^1(X/S)$. 
For example, if $D$ is nef, then $N(D) = 0$.

A similar statement to the following lemma 
for the case of the $\mathbf{R}$-linear systems is proved in 
\cite{BCHM} using a result in \cite{Nakayama}.
We believe that our approach is more natural and easier.

\begin{Lem}~\label{exceptional}
Let $f: (X,B) \to S$ be a projective morphism from a $\mathbf{Q}$-factorial 
LC pair, and let $\alpha: (X,B) \dashrightarrow (Y,C)$ be a minimal model.
Then the prime divisors contracted by $\alpha$ are precisely the 
irreducible components of the numerically fixed part of $K_X+B$.
\end{Lem}

\begin{proof}
Let $p: \tilde X \to X$ and $q: \tilde X \to Y$ be a common resolution.
Then $p^*(K_X+B) - q^*(K_Y+C)$ is effective, and the coefficients of the
exceptional divisors of $\alpha$ are positive.

Since $K_Y+C$ is nef, there is no numerically fixed part;
\[
\inf\{D'_Y \,\vert\, D'_Y \equiv K_Y+C+\epsilon A_Y, D'_Y \ge 0\} = 0
\]
for any $\epsilon > 0$ and any ample divisor $A_Y$ on $Y$.
Because of the negativity of $K_X+B$ for $\alpha$, the difference 
$(K_X+B+\epsilon \alpha_*^{-1}A_Y) - p_*q^*(K_Y+C+\epsilon A_Y)$
is still effective if $\epsilon$ is small enough, where 
$\alpha_*^{-1}A_Y$ is the strict transform of $A_Y$.
Then it follows that
\[
\inf\{D' \,\vert\, D' \equiv K_X+B+\epsilon \alpha_*^{-1}A_Y, D' \ge 0\}
\]
is exceptional for $\alpha$, and so are any irreducible components 
of the numerically fixed part of $K_X+B$. 

Conversely, if $D' \equiv K_X+B+\epsilon A$, then 
$\alpha_*D' \equiv K_Y+C+\epsilon \alpha_*A$.
Hence $D' - p_*q^*\alpha_*D'$ has positive coefficients on any exceptional
divisors of $\alpha$ if $\epsilon$ is small enough.
Therefore any exceptional divisors of $\alpha$ are numerically fixed for 
$K_X+B$.
\end{proof}

%%%%%%%%%%%%%%%%%%%%%%%%%%%%%%%%%%%%%%%%%%%%%%%%%%%%%

\section{Cone decompositions}

A {\em polytope} is a closed convex hull of finitely many 
points in a real vector space.
It is called {\em rational} if these points have rational coordinates.

We start with a polytope decomposition with respect to the canonical models:

\begin{Thm}\label{cano}
Let $(X,\bar B)$ be a $\mathbf{Q}$-factorial 
KLT pair with a projective morphism $f: X \to S$ to a 
base space, $B_1,\dots,B_r$ effective $\mathbf{Q}$-Cartier divisors, 
and $\tilde V$ a polytope in the space
$\{B = \sum_i b_iB_i \,\vert\, b_i \in \mathbf{R}\} \cong \mathbf{R}^r$
such that the pairs $(X,B)$ are LC for all $B \in \tilde V$.
We consider a closed convex subset
\[
V = \{B \in \tilde V \,\vert\, [K_X+B] \in \text{Psef}(X/S)\}.
\]
Assume that for each $B \in V$, there exist a minimal model 
$\alpha: (X,B) \dashrightarrow (Y,C)$ and a canonical model
$g: Y \to Z$ for $f: (X,B) \to S$.
Moreover assume that there exists a real number $\epsilon > 0$ 
for each given $B \in V$ with $\alpha: X \dashrightarrow Y$
and $g: Y \to Z$ as above, such that the morphism 
$g: (Y,\alpha_*B') \to Z$ for $B' \in \tilde V$ 
has minimal and canonical models whenever 
$[K_Y+\alpha_*B'] \in \text{Psef}(Y/Z)$ 
and $\Vert B' - B \Vert \le \epsilon$, where $\Vert \Vert$ denotes 
the maximum norm of the coefficients.
Then there exists a finite decomposition to disjoint subsets 
\[
V = \coprod_{j=1}^s V_j
\]
and rational maps $\beta_j: X \dashrightarrow Z_j$
which satisfies the following conditions:

\begin{enumerate}
\item $B \in V_j$ if and only if $\beta_j$ gives the canonical model 
for $f: (X,B) \to S$.

\item The closures $\bar V_j$, hence $V$, are polytopes for all $j$.
Moreover, if $\tilde V$ is a rational polytope, then
so are the $\bar V_j$ and $V$.
\end{enumerate}
\end{Thm}

\begin{proof}
We use an idea of Shokurov \cite{Shokurov}.
We proceed by induction on the dimension of $V$.
If $\dim V=0$, then the assertion is clear.
Assume that $\dim V>0$.
We fix an arbitrary point $B_0 \in V$.
Let $\alpha_0: (X,B_0) \dashrightarrow (Y_0,C_0)$ be a minimal model and 
$g_0: (Y_0,C_0) \to Z_0$ a canonical model.
We can write $K_{Y_0}+C_0 = g_0^*H_0$ for an $\mathbf{R}$-Cartier divisor 
$H_0$ on $Z_0$ which is ample over $S$.

We take a real number $\epsilon > 0$ such that 
$g_0: (Y_0,\alpha_{0*}B) \to Z_0$ has a minimal model
$\alpha: (Y_0,\alpha_{0*}B) \dashrightarrow (Y,C)$
and a canonical model $g: (Y,C) \to Z$ with $h: Z \to Z_0$ if 
$K_{Y_0}+\alpha_{0*}B$ is pseudo-effective over $Z_0$ and 
$\Vert B - B_0 \Vert \le \epsilon$.
If $\epsilon$ is sufficiently small, then $K_X+B$ is negative for $\alpha_0$.

We can write $K_Y+C = g^*H$ for an $\mathbf{R}$-Cartier divisor $H$ 
on $Z$ which is ample over $Z_0$.
If we take $\delta > 0$ sufficiently small, then 
$(1-\delta)h^*H_0 + \delta H$ is ample over $S$.
If we set $B' = (1-\delta)B_0 + \delta B$, then 
$\alpha \circ \alpha_0: (X,B') \dashrightarrow (Y,C')$ 
for $C' = (1-\delta)\alpha_*C_0 + \delta C$is a
minimal model for $f: (X, B') \to S$ because the negativity still holds, 
and $g: (Y,C') \to Z$ is a canonical model because the ampleness holds.

We take a polytope $\tilde U$ inside $\tilde V$ which contains $B_0$ in the
relative interior and which is contained in the above $\epsilon$ neighborhood.
We take $\tilde U$ to be rational when $\tilde V$ is rational.
Let 
\[
U = \{B \in \tilde U \,\vert\, [K_{Y_0}+\alpha_{0*}B] \in
\text{Psef}(Y_0/Z_0)\}
\]
Then $V \cap \tilde U \subset U$.
By the induction assumption, the boundary $\partial U$ is 
a union of polytopes, and there is a finite polytope decomposition
of $\partial U$ which corresponds to the classification of
canonical models of $g_0: (Y_0,\alpha_{0*}B) \to Z_0$.
Moreover it is rational if so is $\tilde V$.
It also follows that $U \subset V$. 

If $\tilde U$ was chosen sufficiently small, 
then the cones over these polytopes 
with the vertex $B_0$ give a finite polytope decomposition of $U$, 
which is rational if $\tilde V$ is.
Since $V$ is compact, it is covered by finitely many such $U$'s, 
and the assertion is proved.
\end{proof}

Next we consider a further polytope decomposition with respect to the 
minimal models:

\begin{Thm}\label{mini}
Assume the conditions of the above theorem.
Then each $V_j$ has further decomposed to a finite disjoint union
\[
V_j = \coprod_{k=1}^t W_{j,k}
\]
which satisfies the following conditions: 
let $\alpha: X \dashrightarrow Y$ be a birational map such that
\[
W = \{B \in V \,\vert\, \alpha \text{ is a minimal model for } (X,B)\}
\]
is non-empty. Then

\begin{enumerate}
\item There exists an index $j$ such that $W \subset \bar V_j$.

\item If $W \cap V_j$ is non-empty for some $j$, 
then $W \cap V_j$ coincides with one of the $W_{j,k}$.

\item The closure $\bar W_{j,k}$ is a polytope for any $j$ and $k$.
Moreover, if $\tilde V$ is a rational polytope, then
so are the $\bar W_{j,k}$.
\end{enumerate}
\end{Thm}

We note that there may be infinitely many $W$'s such that 
$W \cap V_j=W_{j,k}$ for fixed $j,k$.

\begin{proof}
1. If $\alpha_*(K_X+B_i)$ is nef for $i = 1,2$, then so is 
$\alpha_*(K_X+tB_1+(1-t)B_2)$ for $t \in [0,1]$.
The negativity descends as well, hence $W$ is a convex subset of $V$.

We take a point $B \in W$ in the relative interior, and let $g: Y \to Z$
be the canonical model for $(Y,\alpha_*B)$.
Then $g^*\text{Nef}(Z/S)$ is a face of $\text{Nef}(Y/S)$, and 
$[\alpha_*(K_X+B)] \in g^*\text{Amp}(Z/S)$.
Since $[\alpha_*(K_X+B')] \in \text{Nef}(Y/S)$ for all $B' \in W$, 
it follows that $[\alpha_*(K_X+B')] \in g^*\text{Amp}(Z/S)$ if 
$B'$ is a point in the relative interior.
Therefore $W \subset \bar V_j$ if $V_j$ corresponds to $g \circ \alpha$.

2. Let $\alpha_i: X \dashrightarrow Y_i$ for $i=1,2$ be birational maps,
and let $W_i$ be the corresponding subsets of $V$ which are assumed to
be non-empty.
Assume that there are morphisms $g_i: Y_i \to Z$ such that
$\beta = g_1 \circ \alpha_1 = g_2 \circ \alpha_2$ corresponds to 
some $V_j$.
Let $\gamma: Y_1 \dashrightarrow Y_2$ be the birational map such that 
$\alpha_2 = \gamma \circ \alpha_1$.
We claim that, if $\gamma$ is an isomorphism in codimension one, then 
$W_1 \cap V_j = W_2 \cap V_j$.

Indeed, if $B \in W_1 \cap V_j$, 
then $K_{Y_1}+\alpha_{1*}B = g_1^*H$ for some $H$ on $Z$ 
which is ample over $S$.
Since $\gamma$ is an isomorphism in codimension one, 
it follows that $K_{Y_2}+\alpha_{2*}B = g_2^*H$, 
hence $K_{Y_2}+\alpha_{2*}B$ is nef and equivalent to $K_{Y_1}+\alpha_{1*}B$.
The negativity for the exceptional divisors 
holds at the same time for $Y_1$ and $Y_2$.
Therefore $B \in W_2 \cap V_j$.

Let $\{E_m\}$ be the set of all the prime divisors $E_m$ each of which is 
contained in the numerically fixed part of $K_X+B'$ for $B'$ being a
vertex of $\bar V_j$.
By Lemma~\ref{exceptional}, this is a finite set.
Any numerically fixed component of $K_X+B$ for arbitrary $B \in V_j$
belongs to this set.
Therefore there are only finitely many possibilities 
for the set of prime divisors which coincides 
with the set of numerically fixed components 
of $K_X+B$ for some $B \in V_j$.
Hence the decomposition is finite.

3. The negativity condition of a prime divisor 
with respect to a birational map 
is expressed by a linear inequality with rational coefficients.
Thus each chamber $W_{j,k}$ is bounded by linear equalities 
with rational coefficients of the following two types:
(a) those coming from the negativity condition 
for the corresponding minimal model in the case where the boundary is open, 
and (b) those from other minimal models in the case of closed boundaries.
Hence we have the assertion.
\end{proof}

\begin{Cor}
A sequence of flips in the MMP with scaling terminates
if minimal and canonical models exist along the line segment in the
space of divisors corresponding to this MMP process.
\end{Cor}

\begin{proof}
The chambers $V_j$ are finite in number.
Therefore we may fix one $V_j$ in order to prove the termination.
Let $\alpha: X^- \dashrightarrow X^+$ be a flip 
in the MMP for a pair $(X,B)$ with scaling $H$, and let $t$ be 
the real number such that $K_{X^-}+B+tH$ is numerically trivial
for the flip.

Suppose that 
both $(X, B + (t+\epsilon)H)$ and $(X, B + (t-\epsilon)H)$ belong to
the same chamber $V_j$.
Thus $X^-$ and $X^+$ are minimal models respectively for these pairs.
Let $g^{\pm}: X^{\pm} \to Z$ be the canonical models.
There are ample $\mathbf{R}$-divisors $H^{\pm}$ on $Z$ such that
$K_{X^{\pm}} + B + (t \mp \epsilon)H = (g^{\pm})^*H^{\pm}$.
Since $\alpha$ is an isomorphism in codimension one,
it follows that $K_{X^{\mp}} + B + (t \mp \epsilon)H = (g^{\mp})^*H^{\pm}$.
In other words, $K_{X^{\pm}} + B + (t \pm \epsilon)H$ is nef.
But this is a contradiction.
Therefore there does not exists a flip inside the same chamber $V_j$.
\end{proof}

\begin{Cor}
Let $f: (X,B) \to S$ be a projective morphism from a KLT pair, and
let $\alpha_i: (X,B) \dashrightarrow (Y_i,C_i)$ ($i=1,2$) be two 
minimal models.
Assume that there exists their canonical model.
Then they are connected by a sequence of flops. 
\end{Cor}

We note that the boundary $B$ may not be a $\mathbf{Q}$-divisor. 

\begin{proof}
The $Y_i$ are isomorphic in codimension one by 
Lemma~\ref{exceptional}.
Let $g_i: Y_i \to Z$ be the morphisms to the canonical model.
Let $H_i$ be ample $\mathbf{R}$-divisors on the $Y_i$.
Let $V$ be the triangle spanned by $\mathbf{R}$-divisors $0$, $H_1$ and $H_2$
on $Y_1$, where we use the same symbol for the strict transforms as usual.
We may assume that the $(Y_1,H_i)$ are KLT.

There are finitely many subsets $V_j$ of $V$ corresponding to the 
canonical models by the theorem.
If we replace the $H_i$ by $\epsilon H_i$ for sufficiently small $\epsilon$,
then we may assume that all the $\bar V_j$ contain $0$.
Then the MMP with scaling corresponding to the line segment joining 
$K_{Y_1}+H_1$ and $K_{Y_1}+H_2$ gives the desired sequence of flops.
\end{proof}

We close this paper with some examples in order to illustrate the theorems.

\begin{Expl}
Let $f_0: X_0 \to S$ be a ruled surface with two disjoint sections $S_0$ and 
$S_1$ over a curve, $\alpha_0: X \to X_0$ a 
blowing up at a point on $S_0$, and $f: X \to S$ the composite map.
There is one singular fiber of $f$ which will be denoted by $C_0 + C_1$, where 
$C_0$ is the exceptional curve for $\alpha_0$.

Let $V = \{B = \frac 12 C_0 + tC_1 \,\vert, 0 \le t \le 1\}$.
Then the canonical models for the $(X,B)$ are reduced to the identity to $S$ 
for all $B \in V$.
Let $W_-$, $W_0$ and $W_+$ be the subsets of $V$ defined by equations 
$t < \frac 12$, $t=\frac 12$ and $t > \frac 12$, respectively.
They correspond to minimal models $\alpha_0: X \to X_0$, $X \to X$, and
$\alpha_1: X \to X_1$, the contraction morphism of $C_1$.
\end{Expl}

\begin{Expl}
Let $X^-$ be a Calabi-Yau threefold and let
$\alpha: X^- \to Y \leftarrow X^+$ be a 
flop such that there is a non-canonical isomorphism $h: X^- \cong X^+$ 
as in Example~\ref{223}.
Let $H^-$ be an ample divisor on $X^-$, and let $H^+ = \alpha_*H^-$
be the strict transform which is also considered to be a divisor on $X^-$
by using $h$.
If $\epsilon > 0$ is sufficiently small, then $(X^-,B_t)$ for 
$B_t = \epsilon((1-t)H^- + tH^+)$ with $t \in [0,1]$ are KLT. 
We have a decomposition $V = V^- \coprod V^0 \coprod V^+$, where 
$V = [0,1]$, $V^- = [0,1/2)$, $V^0=\{1/2\}$ and $V^+=(1/2,1]$, corresponding
to the canonical models given by $X^-$, $Y$ and $X^+$.
The chambers $W^-=[0,1/2]$ and $W^+=[1/2,1]$ correspond to minimal 
models $X^-$ and $X^+$.
When $t=1/2$, the canoical model $Y$ has two minimal models $X^{\pm}$.
\end{Expl}

We recall two examples from \cite{CY}~Example~3.8 
in which there are infinitely 
many cones in the space of numerical equivalence classes of 
$\mathbf{R}$-divisors.
In this sense we can say that our theorems deal with only \lq\lq local'' 
situations.

\begin{Expl}
Let $f: X \to S$ be the versal deformation of a singular fiber 
of type $I_2$ on an elliptic surface, i.e.,
a curve $C$ which has two
irreducible components isomorphic to $\mathbf{P}^1$ 
intersecting transversally at two points.
We have $\dim X=3$ and $\dim N^1(X/S) = 2$.

There are coordinates on $N^1(X/S)$ defined by $x = (D \cdot C_1)$ and 
$y=(D \cdot C_2)$ where the $C_i$ are irreducible components of $C$.
The pseudo-effective cone is the half space defined by 
$x+y \ge 0$.
The effective cone is not closed because points $(1,-1)$ and $(-1,1)$ are 
not represented by any effective $\mathbf{R}$-divisors.

The rays generated by $(n+1,-n)$ for $n \in \mathbf{Z}$ divide the 
pseudo-effective cone into infinitely many polytopes which 
correspond to canonical models.
\end{Expl}

\begin{Expl}\label{223}
Let $X$ be a generic hypersurface of degrees $(2,2,3)$ in 
$\mathbf{P}^1 \times \mathbf{P}^1 \times \mathbf{P}^2$.
The pseudo-effective cone is the cone over the convex hull of the points 
$A = (0,0,1)$ and $C_n = (n+1, -n, \frac 32n(n+1))$ for $n \in \mathbf{Z}$, 
because its boundary points correspond to fibrations to 
lower dimensional spaces, i.e., elliptic fibrations and $K3$ fibrations.
The effective cone is closed because the points $A$ and $C_n$ are 
represented by effective divisors.
But it is divided into infinitely many polytopes again.
The projection $X \to S = \mathbf{P}^1 \times \mathbf{P}^1$ induces a 
linear map $N^1(X) \to N^1(X/S)$ given by $(x,y,z) \mapsto (x,y)$,
and the effective cone in the previous example is the image under this map.
\end{Expl}

Department of Mathematical Sciences, University of Tokyo,

Komaba, Meguro, Tokyo, 153-8914, Japan

kawamata@ms.u-tokyo.ac.jp

\end{document}